\documentclass[preprint]{imsart}
\usepackage{amsmath,amssymb,amsthm,comment}

\newcommand{\levy}{L\'{e}vy }
\newcommand{\p}{{\mathbb P}}
\newcommand{\e}{{\mathbb E}}
\newcommand{\D}{{\mathrm d}}

\newcommand{\R}{{\mathbb R}}
\newcommand{\1}[1]{\mbox{\rm  1}_{\{#1\}}}
\newcommand{\tx}[1]{\tau_{\{#1\}}}
\renewcommand{\a}{{\alpha}}

\newcommand{\bs}{\boldsymbol}

\newcommand{\one}{{\bs 1}}
\newcommand{\matI}{\mathbb{I}}
\newcommand{\matO}{\mathbb{O}}
\newcommand{\bpi}{{\bs \pi}}
\newcommand{\Dpi}{\Delta_\bpi}

\newtheorem{theorem}{Theorem}
\newtheorem{lemma}{Lemma}
\newtheorem{prop}{Proposition}
\newtheorem{cor}{Corollary}

\begin{document}

\begin{frontmatter}

\title{Potential measures of one-sided Markov additive processes with reflecting and terminating barriers}
\runtitle{Potential measures of MAPs}

\begin{aug}
  \author{\fnms{Jevgenijs}  \snm{Ivanovs}\ead[label=e1]{jevgenijs.ivanovs@unil.ch}},

  \runauthor{J. Ivanovs}

  \affiliation{Department of Actuarial Science, Faculty of Business and Economics, University of Lausanne}

  \address{University of Lausanne, Quartier UNIL-Dorigny, 1015 Lausanne, Switzerland,\\ 
          \printead{e1}}
\end{aug}

\begin{abstract}
Consider a one-sided Markov additive process with an upper and a lower barrier, where each can be either reflecting or terminating.
For both defective and non-defective processes and all possible scenarios we identify the corresponding potential measures,
which generalizes a number of results for one-sided \levy processes.
The resulting rather neat formulas have various applications, and in particular they lead to quasi-stationary distributions of the corresponding processes.
\end{abstract}



\end{frontmatter}

\maketitle

\section{Introduction}
A Markov additive process (MAP) in continuous time is a natural generalization of a \levy process with various applications in queueing, risk theory and financial mathematics, see e.g.~\cite{APQ}.
It can be seen as a \levy process in Markov environment, which provides rich modeling possibilities.

The main exit problems for spectrally negative MAPs are solved in~\cite{ivanovs_scale}, including one-sided and two-sided exit, as well as exit for reflected processes. 
Following these developments we consider a spectrally negative MAP with two barriers (upper and lower), where each can be either terminating or reflecting, 
and identify the corresponding potential measures (also known as resolvent measures). 
It is assumed that a process is stopped (killed) upon its passage over a terminating barrier, i.e.\ above an upper or below a lower terminating barrier.
Also it is allowed to place barriers at $\infty$ and $-\infty$, which corresponds to a model with less than two barriers.

Potential measures have numerous applications as can be anticipated from the general theory of Markov processes. 
In the case of spectrally negative \levy processes various potential measures can be given in an explicit form, see~\cite[Sec.\ 8.4]{kyprianou} and references therein.
It should also be noted that potential measures and their densities readily lead to the distribution of the corresponding process at an exponential time; 
more can be done in the MAP setting which will be discussed in the following. 
Recent progress in the theory of spectrally negative MAPs allows us to present rather neat formulas and short proofs.

We formulate the problem and discuss the relation between potential measures and quasi-stationary distributions in Section~\ref{sec:prelim}.
Section~\ref{sec:exit} reviews some basic results from the exit theory of MAPs, and defines related matrices and matrix-valued functions.
Occupation densities, reviewed in Section~\ref{sec:occ_dens}, constitute the main tool for deriving potential densities in case of terminating barriers, see Section~\ref{sec:terminating}.
Time-reversal, discussed in Section~\ref{sec:time}, is the main tool used in Section~\ref{sec:reflecting} in deriving potential densities for reflected processes.
The above results are obtained for defective processes, which are extended to non-defective ones in Section~\ref{sec:non_def}. Some concluding remarks are given in Section~\ref{sec:remarks}.

The results on potential measures (and hence on quasi-stationary distributions) are spread among Theorem~\ref{thm:absorbing}, Theorem~\ref{thm:reflecting}, Corollary~\ref{cor:reflecting} and Corollary~\ref{cor:limits}.
In addition, Corollary~\ref{cor:lim_distr} presents stationary distributions. Finally, this paper contains some useful identities concerning time-reversed process and various limits.
This paper complements~\cite{runhuan}, where some of the potential measures were identified using a very different approach.
Some formulas may look remarkably different, which is a known issue with MAPs.

\section{Preliminaries}\label{sec:prelim}
A MAP is a bivariate Markov process $(X,J)=\{(X(t),J(t)):t\geq 0\}$ where $X$ is an additive level component and $J$ represents the environment, see~\cite[Sec.\ XI.2a]{APQ}.
It is assumed that $(X,J)$ is adapted to some right-continuous, complete filtration $\{\mathcal F_t:t\geq 0\}$, 
and that $J$ is an irreducible Markov chain on a finite number of states, say $n$.
The defining property of a MAP states that for any $t\geq 0$ and any $i=1,\ldots,n$ conditionally on $\{J(t)=i\}$ the process
\[\{(X(t+s)-X(t),J(t+s)):s\geq 0\}\]
is independent of $\mathcal F_t$ and has the law of $\{(X(s)-X(0),J(s)):s\geq 0\}$ given $\{J(0)=i\}$. 
In other words, the environment governs the increments of the level process. Importantly, this property also holds when $t$ is replaced by any stopping time $\tau$;
note that $\{J(\tau)=i\}$ implicitly assumes that \mbox{$\tau<\infty$}.
We allow for defective (or killed) processes, i.e.\ we add an additional absorbing cemetery state to the environment without writing it explicitly.
Hence $\{J(t)=i\}$ also means that the process has survived up to time $t$. In the special case when $n=1$ one obtains a \levy process, which may be killed at an independent exponential time.

Putting to work Markov additive property requires repeated conditioning on the state of the environment, which leads to matrix algebra and justifies the following notation.
For a random variable $Y$ we write $\e_x[Y;J(t)]$ to denote an $n\times n$ matrix with $ij$-th element
\[\e_{x,i}(Y;J(t)=j)=\e(Y\1{J(t)=j}|X(0)=x,J(0)=i).\]
Similarly, for an event $A$ we write $\p_x[A,J(t)]=\e_x[{\rm 1}_A;J(t)]$ for the corresponding matrix of probabilities. Moreover, subscript $x$ can be omitted when $x=0$. 
Finally, the identity and the zero matrices are denoted by $\matI$ and $\matO$ respectively, and $\Delta_{\bs v}$ stands for a diagonal matrix with a vector $\bs v$ on the diagonal.

We can assume that $X(0)=0$ and the barriers are placed at $-a$ and at $b$, where
$a,b\in [0,\infty]$ (not simultaneously 0), because otherwise one can simply shift the picture.
We write $|-a,b|,|-a,b],[-a,b|,[-a,b]$ to denote different scenarios, where $|$ means termination and $[$ or $]$ mean reflection at the corresponding barriers.
In particular, $X_{[-a,b|}(t)$ is a process reflected at $-a$ and terminated at $b$, i.e.\ upon exiting the interval $(-\infty,b]$.
In the case when a barrier is placed at $-\infty$ or $\infty$, i.e.\ there is no barrier, we write $(-\infty,\infty),(-\infty,b|,(-\infty,b],|-a,\infty),[-a,\infty)$ 
to denote all possible scenarios, where the first corresponds to a free process, the second means termination upon exiting the interval $(-\infty,b]$, and so forth.
For a rigorous definition of reflection, both one-sided and two-sided, see e.g.~\cite{APQ,kella_reflecting}. 
Here we only recall the simplest cases $X_{[0,\infty)}(t)=X(t)-\underline X(t)$ and $X_{(-\infty,0]}(t)=X(t)-\overline X(t)$, 
where $\underline X(t)$ and $\overline X(t)$ are the running infimum and supremum respectively,
and note that any reflecting barrier locally acts in a similar way, c.f.~\cite[Ch.\ XIV.3]{APQ}.

For each of the above scenarios $I$ (e.g.\ $I=|-a,b]$) we consider the corresponding potential measure (a matrix of measures)
\begin{align}\label{eq:potential}
 &U_{I}(A)=\int_0^\infty \p[X_I(t)\in A,J(t)]\D t=\left(\e_i\int_0^\infty \1{X_I(t)\in A,J(t)=j}\D t\right),
\end{align}
where $A\in\mathcal B[-a,b]$. It turns out that in all the cases the measure $U_I(A)$ has a density $u_I(x)$ on $(a,b)$ with respect to Lebesgue measure. 
We identify this density and compute the point masses at $-a$ and $b$.
It will be shown that there is never a point mass at the level of a terminating barrier, and hence we only specify the point masses for the levels of reflecting barriers.
One of the possible uses of potential measures is given by the following basic formula
\begin{align}\label{eq:potential_appl}
 \e_i\int_0^\infty f(X(t),J(t))\D t=\sum_{j}\int_{\mathbb R}f(x,j)U_{ij}(\D x),
\end{align}
where $f\geq 0$ is a measurable function, which is equal to 0 for the cemetery state of the environment.

Recall that we allow for defective MAPs, in which case one can think that the time runs only up to the killing epoch (when $J$ enters the cemetery state).
A defective MAP can be seen as a non-defective MAP killed at some rate $q_i\geq 0$ while $J$ is in $i$ for all $i=1,\ldots,n$.
A special case arises when $q_i=q$ for all $i$, i.e.\ a MAP is killed at an independent exponential time $e_q$ of rate $q$.
Then
\begin{equation}\label{eq:intro1}
 \p[X_I(e_q)\in A,J(e_q)]=q\int_0^\infty e^{-qt}\p[X_I(t)\in A,J(t)]\D t=qU^q_{I}(A),
\end{equation}
where superscript $q$ is used to distinguish between objects corresponding to a MAP and its killed version.
This relation can be generalized in the following way. 
Let ${\bs q}=(q_1,\ldots,q_n)$ be a vector of killing rates and let $T$ be the killing time, i.e.\ the life-time of the Markov chain $J$, implying that $T$ has a corresponding phase type distribution.
Then a standard argument shows that 
\[\p[X_I(T)\in A,J(T)]=\int_0^\infty \p^{\bs q}[X_I(t)\in A,J(t)]\D t\Delta_{\bs q}=U^{\bs q}_{I}(A)\Delta_{\bs q}.\]
One can think about an independent Poissonian observer whose rate $q_i$ depends on the environment, then $U^{\bs q}_{I}(A)\Delta_{\bs q}$
gives the distribution of the process at his first observation epoch~$T$. 

\section{Exit theory review}\label{sec:exit}
Throughout this paper it is assumed that the level component $X$ has no positive jumps, 
and that for each $i$ the process $X$ given $\{J(0)=i\}$ visits $(0,\infty)$ before the switch of the environment with positive probability 
(none of the underlying \levy processes is a downward subordinator). The second assumption also appearing in~\cite{ivanovs_scale} allows to greatly simplify our
notation and to avoid unpleasant technical difficulties.
Let us briefly review the exit theory of such MAPs.

For $x\geq 0$ define the first passage times
\begin{align*}
&\tau_x^+=\inf\{t\geq 0:X(t)>x\}, &\tau_{-x}^-=\inf\{t\geq 0:X(t)<-x\}.
\end{align*}
In addition we define the first hitting time of a level $x\in\mathbb R$ by
\[\tx{x}=\inf\{t>0:X(t)=x\}.\]
It is known that $\tau_x^+\leq \tx{x}$ and $\tau_{-x}^-\leq \tx{-x}$ for $x\geq 0$ a.s., i.e.\ $X$ can not hit a level without immediately passing over it,
which can be seen from the small-time behavior of \levy processes, see also~\cite[Prop.\ 7]{ivanovs_scale}. 

There is a matrix-valued function $F(\a)$, which characterizes the law of the process $(X,J)$, and in particular $\e[e^{\a X(t)};J(t)]=e^{F(\a)t},\a,t\geq 0$. 
Note that $F(0)$ is the transition rate matrix of $J$, and hence our MAP is non-defective if and only if $F(0)\bs 1=\bs 0$, where $\bs 1$ and $\bs 0$ denote column vectors of 1s and 0s respectively.
In this case let a row vector $\bpi$ be the stationary distribution of~$J$, and let $\mu=\e_\bpi X(1)$ be the asymptotic drift.
According to $\mu< 0,\mu>0$ and $\mu=0$ the process $X(t)$ (as $t\rightarrow\infty$) tends to $-\infty$, tends to~$\infty$, and oscillates between $-\infty$ and $\infty$, see~\cite[Prop.\ XI.2.10]{APQ}. 

It is not difficult to see that $\{J(\tau_x^+),x\geq 0\}$ is a Markov chain with some transition rate matrix $G$, so that 
\[\p[J(\tau_x^+)]=e^{G x}, x\geq 0.\]
It is known that $G$ is a right solution (unique in a certain class) to a matrix integral equation $F(-G)=\matO$, see~\cite{breuer} for details.
We use the symbol $G$ to be consistent with the theory of discrete time skip-free upwards MAPs, see~\cite[Sec.\ XI.3]{APQ}.
In fact, there are numerous similarities between discrete and continuous time theories, which are explored in~\cite{links_discrete}.
Analogously to the discrete time case, one defines a matrix $R$ as the unique left solution to $F(-R)=\matO$, but see also~\eqref{eq:R}.
It is known that both $G$ and $R$ are non-singular, unless the MAP is non-defective and $\mu\geq 0$, in which case they both have a simple eigenvalue at 0.

Another important object in the exit theory for MAPs is a matrix-valued function $W(x)$, which is continuous on $[0,\infty)$ and is identified by the transform 
\[\int_0^\infty e^{-\a x}W(x)\D x=F(\a)^{-1}\] for large enough $\a$; we put $W(x)=\matO$ for $x<0$.
It holds that $W(x)$ is non-singular for $x>0$ and 
\begin{equation}\label{eq:Wexit}
 \p[\tau_b^+<\tau_{-a}^-;J(\tau_b^+)]=\p[\tau_b^+<\tx{-a};J(\tau_b^+)]=W(a)W(a+b)^{-1}.
\end{equation}
The reasons for the first equality are the following. 
Firstly $X$ can not hit a level $-a$ without immediately going below it, and secondly if $X$ jumps below $-a$ then it has to hit it on the way to $b$.

Let us discuss differentiability of $W(x)$ for $x>0$. 
In~\cite[Thm.\ 5]{ivanovs_scale} it is shown that both left and right derivatives, $W_-'(x)$ and $W_+'(x)$ respectively, exist for all $x>0$, but may not coincide at countably many points. 
Furthermore, this result can be used to show that $\lim_{y\uparrow x}W_+'(y)=W_-'(x)$, and that $W'_+(x)$ is Riemann integrable on any interval $[0,a]$. 
When $W'_+(x)$ enters into density representation, we simply write $W'(x)$.

Define yet another function
\[Z(x)=\matI-\int_0^xW(y)\D y F(0),\]
which is used in various identities, and in particular in
\begin{equation}\label{eq:exit_minus}
 \p[J(\tau_{-a}^-)]=Z(a)-W(a)R^{-1}F(0),
\end{equation}
where in the case of a non-defective process with $\mu\geq 0$ the term $R^{-1}F(0)$ should be interpreted in a limiting sense (by letting the killing rates approach~0), 
see~\eqref{eq:R} and~\cite{ivanovs_scale}.
In general, the function $Z$ is defined as a function of two parameters $\a$ and~$x$, but for the purpose of this paper the case of $\a=0$ is sufficient.

Let us finally remark about simplifications when one considers a \levy process, i.e.\ $J$ lives on a single state.
Firstly, all the matrices become scalars, and in particular $F(0)=-q$, where $q\geq 0$ is the killing rate.
Moreover, $G=R=-\Phi$, where $\Phi\geq 0$ is called the right-inverse of the Laplace exponent $F(\a)$, i.e.\ it solves $F(\Phi)=0$.
More precisely, $\Phi=0$ when $q=0$ and $\mu\geq 0$, and otherwise $\Phi$ is a unique positive zero of~$F(\a)$.

\section{Occupation densities}\label{sec:occ_dens}
Let us consider $L(x,j,t)$ the occupation density of $(X,J)$ at $(x,j)$ up to time $t$, which is defined in~\cite{ivanovs_scale} based on corresponding theory for \levy processes.
Firstly, $L(x,j,t)$ always exists under our assumptions. It is nonnegative, measurable in $x$, and for all $t\geq 0$ and all measurable $f\geq 0$ satisfies
\begin{equation}\label{eq:occ_density}
 \int_0^t f(X(s),J(s))\D s=\sum_j\int_{\mathbb R}f(x,j)L(x,j,t)\D x \text{ a.s.},
\end{equation}
where $f$ is 0 for the cemetery state. Compare this occupation density formula to~\eqref{eq:potential_appl}. 
Note that $L$ leads to more refined formula, but that is valid only for a free process and moreover $L$ lacks explicit expression.

Importantly, for every $x\in\mathbb R$ and every $j$: $\{L(x,j,t):t\geq 0\}$ is $\mathcal F_t$ adapted process, which increases only when $X=x,J=j$, and inherits the following additive property.
Consider a stopping time $\tau$ such that $X(\tau)=y\in\mathbb\R$ on $\{J(\tau)=i\}$ for some $y$ and $i$. 
Then on the event $\{J(\tau)=i\}$ the shifted process $\{L(x,j,\tau+s)-L(x,j,\tau):s\geq 0\}$ is independent of $\mathcal F_\tau$ and has the law of $\{L(x-y,j,s):s\geq 0\}$ given $\{J(0)=i\}$.

Let us define another important matrix-valued function $H(x)$ by \[H_{ij}(x)=\e_i L(0,j,\tau_x^+) \text { for }x>0\] and $H(0)=\lim_{x\downarrow 0}H(x)$.
In other words, $H(x)$ is a matrix of expected occupation densities at 0 for the time interval $[0,\tau_x^+)$. 
The following relation is obtained in~\cite{ivanovs_scale}
\begin{equation}\label{eq:W}
 W(x)=e^{-G x}H(x)
\end{equation}
for $x\geq 0$.
Let also $H=H(\infty)$ be the expected occupation density at 0. It is known that $H$ has finite entries and is invertible unless the process is non-defective and $\mu=0$.
 
In the rest of this section we suppose that the process is either defective or $\mu\neq 0$. 
Then $H$ has finite entries and by the additive property of $L$ we have for $x\geq 0$
\begin{equation}\label{eq:Hx}
H(x)=H-\p[J(\tau_x^+)]\p[J(\tx{-x})]H.\end{equation}
So combining~\eqref{eq:W} and~\eqref{eq:Hx} we get
\begin{equation}\label{eq:hitting}
 \p[J(\tx{x})]=e^{G x}-W(-x)H^{-1}
\end{equation}
for $x\leq 0$, but then it is clearly true also for $x>0$.
Finally, in~\cite{links_discrete} the following relation between $R$ and $G$ is obtained:
\begin{equation}\label{eq:R}
 R=H^{-1}GH.
\end{equation}

Often it is simpler and more convenient to work with defective processes.
 In order, to retrieve identities for non-defective processes, one can let the killing rates go to 0.
We follow this idea in the rest of the paper and provide comments concerning non-defective processes in Section~\ref{sec:non_def}.

\section{Terminating barriers}\label{sec:terminating}
Let us present results for all the cases when there are no reflecting barriers.
\begin{theorem}\label{thm:absorbing}
 For a defective MAP it holds that
\begin{align}
 u_{(-\infty,\infty)}(x)&=e^{G x}H-W(-x)=He^{R x}-W(-x),\label{thm:11}\\
u_{(-\infty,b|}(x)&=e^{G b}W(b-x)-W(-x),\label{thm:12}\\
u_{|-a,\infty)}(x)&=W(a)e^{R(x+a)}-W(-x),\label{thm:13}\\
u_{|-a,b|}(x)&=W(a)W(a+b)^{-1}W(b-x)-W(-x)\label{thm:14},
\end{align}
where $u_I(x)$ is a density of the measure $U_I(\D x)$ on the corresponding interval, see~\eqref{eq:potential}.
\end{theorem}
In the case of a \levy process this result is due to~\cite{bingham,suprun,bertoin_exp_decay}.
The methods employed in these papers are different from ours. In particular, the last paper relies on the Wiener-Hopf factorization.
It should be noted that for a \levy process $H=1/F'(\Phi)$ yielding
\[u_{(-\infty,\infty)}(x)=\frac{1}{F'(\Phi)}e^{-\Phi x}-W(-x),\]
which agrees with a representation appearing in~\cite[Thm.\ 1]{pistorius_potential}.

\begin{proof}[Proof of Theorem~\ref{thm:absorbing}]
When computing densities $u_I(x)$ we will not consider the case $x=0$, since $u_I(0)$ can be arbitrary.
 The potential measure of the free process $(X,J)$ is given by
\[U_{ij}(A)=\e_i\int_0^\infty \1{X(t)\in A,J(t)=j}\D t=\e_i\int_A L(x,j,\infty)\D x\]
according to~\eqref{eq:occ_density}, and hence it has a density $u_{ij}(x)=\e_i L(x,j,\infty)$.
By the additive property of $L$ and~\eqref{eq:hitting} we write 
\[u(x)=\p[J(\tx{x})]H=e^{G x}H-W(-x).\]
Together with~\eqref{eq:R} this proves~\eqref{thm:11}.

Similarly we find for $x<b$
\begin{align*}
&u_{(-\infty,b|}(x)=\left(\e_i L(x,j,\tau_b^+)\right)=\p[J(\tx{x})]H-\p[J(\tau_b^+)]\p[J(\tx{x-b})]H\\
&=e^{Gx}H-W(-x)-e^{Gb}(e^{G(x-b)}H-W(b-x))=e^{Gb}W(b-x)-W(-x).
\end{align*}
For $x>-a$ we have
\begin{align*}
&u_{|-a,\infty)}(x)=\left(\e_i L(x,j,\tau_{-a}^-)\right)=\p[J(\tx{x})]H-\p[J(\tx{-a})]\p[J(\tau_{a+x}^+)]H=\\
&e^{Gx}H-W(-x)-(e^{-Ga}-W(a)H^{-1})e^{G(x+a)}H=W(a)e^{R(x+a)}-W(-x), 
\end{align*}
where we used~\eqref{eq:R} and the fact that if a process has to pass over $-a$ and then return to $x>-a$ then it has to hit $-a$. 
Finally, for $-a<x<b$ we write
\begin{align*}
&u_{|-a,b|}(x)=\left(\e_i L(x,j,\tau_{-a}^-\wedge\tau_b^+)\right)=\p[J(\tx{x})]H-\\
&\p[\tau_b^+<\tx{-a},J(\tau_b^+)]\p[J(\tx{x-b})]H-\p[\tx{-a}<\tau_b^+,J(\tx{-a})]\p[J(\tau_{x+a}^+)]H.
\end{align*}
Observe that 
\[\p[\tx{-a}<\tau_b^+,J(\tx{-a})]=\p[J(\tx{-a})]-\p[\tau_b^+<\tx{-a},J(\tau_b^+)]\p[J(\tx{-a-b})]\] and use~\eqref{eq:Wexit} and~\eqref{eq:hitting} to obtain
after a number of cancellations
\[u_{|-a,b|}(\D x)=W(a)W(a+b)^{-1}W(b-x)-W(-x).\]
\end{proof}

It is noted that~\eqref{thm:14} can be used to obtain all the other results in Theorem~\ref{thm:absorbing} by taking limits as $a,b\rightarrow\infty$.
Computation of the limiting expressions is often not easy. We go the other way around and use Theorem~\ref{thm:absorbing} to state some useful limits.
\begin{cor}\label{cor:limits}
For a defective process and $y\geq 0$ the following limits hold as $x\rightarrow\infty$
\begin{align*}
 &e^{G x}W(x)\rightarrow H, &W(x)e^{R x}\rightarrow H,\\ 
&W(x)W^{-1}(x+y)\rightarrow e^{G y}, &W^{-1}(x+y)W(x)\rightarrow e^{R y}.
\end{align*} 
\end{cor}
\begin{proof}
The proof of Theorem~\ref{thm:absorbing} can be used to show that 
\[\lim_{a\rightarrow\infty}u_{|-a,\infty)}(x)=u_{(-\infty,\infty)}(x)\text{ for }x\neq 0.\]
Using~\eqref{thm:13} and~\eqref{thm:11} we see that $W(a)e^{R a}\rightarrow H$ as $a\rightarrow\infty$. The other limits are obtained in a similar way.
Let us also remark that two of these limits can be obtained directly: $e^{G x}W(x)=H(x)\rightarrow H$ and $W^{-1}(x+y)W(x)=H^{-1}(x+y)e^{G y}H(x)\rightarrow H^{-1}e^{G y}H=e^{R y}$.
\end{proof}


\section{Time reversal}\label{sec:time}
Consider a non-defective MAP $(X,J)$ and let $\bs\pi$ be the stationary distribution of $J$.
Assuming that $J(0)$ has the distribution $\bs \pi$, one defines for an arbitrary $t>0$ a time-reversed process by 
\begin{align*}
&\hat J(s)=J((t-s)-), &\hat X(s)=X(t)-X((t-s)-),
\end{align*}
where $s\in[0,t)$. It is known that $(\hat J(s),\hat X(s))_{s\in[0,t)}$ is again a spectrally negative MAP (no downward subordinators among its components).
This time-reversed MAP is characterized by 
\begin{align}\label{eq:timerev}
\hat F(\a)=\Dpi^{-1}F(\a)^T\Dpi,\end{align} and it can be continued to $s\in[0,\infty)$.
Now suppose that we kill both the original and the time-reversed processes using the same vector of killing rates~$\bs q$.
Then 
\[\hat F^{\bs q}(\a)=\hat F(\a)-\Delta_{\bs q}=\Dpi^{-1}(F(\a)-\Delta_{\bs q})^T\Dpi=\Dpi^{-1}F^{\bs q}(\a)^T\Dpi,\]
i.e.\ the same identity is true, where $\bpi$ corresponds to the non-defective transition rate matrix $F^{\bs q}(0)+\Delta_{\bs q}$.
In the following we drop the superscript~$\bs q$, where it does not lead to confusion, and use $\hat \p$ to denote the law of the (possibly killed) time-reversed process $(\hat X(s),\hat J(s))$.
Importantly, time-reversed quantities can be expressed through their original analogues.
\begin{prop}\label{prop:time_rev}
 The following identities hold true:
\begin{align}
&\hat R=\Dpi^{-1}G^T\Dpi, &\hat G=\Dpi^{-1}R^T\Dpi,\label{eq:timerev_G_R}\\
&\hat H=\Dpi^{-1}H^T\Dpi, &\hat W(x)=\Dpi^{-1}W(x)^T\Dpi.\label{eq:timerev_W}
\end{align}
\end{prop}
\begin{proof}
 Identities in the first line were obtained in~\cite{links_discrete} by using the fact that $\hat G$ is a unique solution (in a certain class) of $\hat F(-G)=\matO$, and similarly for $\hat R$.
In a similar way we compute using~\eqref{eq:timerev}
\[\int_0^\infty e^{-\a x}\Dpi^{-1}W(x)^T\Dpi\D x=\Dpi^{-1}[F(\a)^{-1}]^T \Dpi=\hat F(\a)^{-1},\]
which establishes the result for $\hat W(x)$, because its transform identifies the continuous $\hat W(x)$ uniquely.
Finally, we consider 
\[\hat W(x)e^{\hat R x}=\Dpi^{-1}W(x)^Te^{G^Tx}\Dpi=\Dpi^{-1}(e^{Gx}W(x))^T\Dpi,\]
which by letting $x\rightarrow\infty$ and using Corollary~\ref{cor:limits} yields $\hat H=\Dpi^{-1}H^T\Dpi$.
The last result is obtained for defective processes, but can be easily extended to non-defective ones by letting the killing rates approach 0.
\end{proof}
Let us complement Corollary~\ref{cor:limits} with another useful limiting result, 
which also demonstrates how time-reversed quantities can be used to change the order of matrix multiplication.
\begin{lemma}\label{lem:limits}
For a defective process the following limits hold as $x\rightarrow\infty$
\begin{align*}
&Z(x)F(0)^{-1}W(x)^{-1}\rightarrow G^{-1}, &W(x)^{-1}Z(x)F(0)^{-1}\rightarrow R^{-1}.
\end{align*}
\end{lemma}
\begin{proof}
Note that $W(x)^{-1}$ stays bounded as $x\rightarrow\infty$, because for $x>1$ the matrix $W(1)W(x)^{-1}$ is a probability matrix.
Now the second limit follows readily from~\eqref{eq:exit_minus}. Using the relation $Z(x)F(0)^{-1}=F(0)^{-1}-\int_0^x W(y)\D y$ we write
\[Z(x)F(0)^{-1}W(x)^{-1}=\Dpi^{-1}(\hat W(x)^{-1}\hat Z(x)\hat F(0)^{-1})^T\Dpi\rightarrow\Dpi^{-1}\hat R^{-1}\Dpi,\]
which is $G^{-1}$. This concludes the proof.
\end{proof}

Let us present an application of time-reversal, which lays a basis for deriving a formula for the potential density in the case of two reflecting barriers.
\begin{lemma}\label{lem:time_rev}
 For a possibly defective $(X(s),J(s))$, any $t>0$ and $x\in[0,a+b]$ it holds that
\begin{align*}
 &\p_{a}[X_{[0,a+b]}(t)\geq x;J(t)]\\
&=\Dpi^{-1}\left(\hat\p[\tau\leq t,X(\tau)\geq x;J(t)]+\hat\p[\tau>t,X(t)+a\geq x;J(t)]\right)^T\Dpi.
\end{align*}
where $\tau=\inf\{s\geq 0:X(s)\notin [x-a-b,x)\}$.
\end{lemma}
This type of
identity was first noted in~\cite{lindley59} in the case of
a random walk with two reflecting barriers. A short derivation of
its continuous-time analogue is given in~\cite[Prop.~XIV.3.7]{APQ}, see also~\cite{loss_rates} for the case of
a Markov additive input. The above results concern stationary distribution. 
Here we allow for a finite time~$t$ and a defective process. 
\begin{proof}[Proof of Lemma~\ref{lem:time_rev}]
 First assume that $(X(s),J(s))$ is non-defective. 
Similarly to the proof of~\cite[Prop.~XIV.3.7]{APQ}, analysis of the sample paths of the original process (and hence of the corresponding time-reversed process) shows that
\begin{align}\label{eq:lem_tr}
 &\{X_{[0,a+b]}(t)\geq x, X(0)=a\}\\&=\{\hat\tau\leq t,\hat X(\hat \tau)\geq x\}\cup\{\hat\tau>t,\hat X(t)+a\geq x\},\nonumber
\end{align}
where $\hat X(0)=0$.
It is easy to understand this relation by considering a free process and shifting the reflecting boundaries as time evolves. 
Then exit of $\hat X(s)$ from $[x-a-b,x)$ through $x$ means that the original process will not be able to shift the boundaries high enough right before time $t$ to make $X_{[0,a+b]}(t)<x$.
In the same way exit over $x-a-b$ implies the converse.
In case of no exit from this interval, $X_{[0,a+b]}(t)\geq x$ if and only if $X(t)\geq x$, because the barriers can not be shifted far enough.
Conclude by observing that $X(t)=\hat X(t)+a$. 

Conditioning on the states of $J$ at times 0 and $t$ we arrive at
\begin{align*}
 &\p_{a,i}(X_{[0,a+b]}(t)\geq x|J(t)=j)\\
&=\hat\p_j(\tau\leq t,X(\tau)\geq x|J(t)=i)+\hat\p_j(\tau>t,X(t)+a\geq x|J(t)=i),
\end{align*}
which immediately yields the result for a non-defective process.

Let us kill the original non-defective process using the killing rates $\bs q$. 
Given the whole evolution of the non-defective process on $[0,t]$ (and hence its time-reversed counterpart), the probability of no killing is 
\[\exp\left(-\sum_i q_i\int_0^t\1{J(s)=i}\D s\right).\]
Note that this expression stays the same if we substitute $\hat J$ instead of $J$.
Finally, intersect both sides of~\eqref{eq:lem_tr} with the event of no killing in $[0,t]$, and condition on the evolution of the non-defective process 
to show that the result is also true in presence of killing.
\end{proof}

\section{Reflecting barriers}\label{sec:reflecting}
\begin{theorem}\label{thm:reflecting}
 For a defective MAP it holds that the potential measure $U_{[-a,b]}(\D x)$ has a density 
\begin{align}\label{eq:thm1}
u_{[-a,b]}(x)=-Z(a)F(0)^{-1}W(a+b)^{-1}W'(b-x)-W(-x)
\end{align}
on $(-a,b)$, and
\begin{align*}
 &U_{[-a,b]}\{-a\}=\matO, &U_{[-a,b]}\{b\}=-Z(a)F(0)^{-1}W(a+b)^{-1}W(0).
\end{align*}
\end{theorem}
\begin{proof}
  Take an independent r.v.\ $T\sim{\rm Exp}(q), q>0$. According to~\eqref{eq:intro1} and Lemma~\ref{lem:time_rev} we write for $x\in[0,a+b]$
\begin{align*}
 &U^q_{[-a,b]}[x-a,b]=\frac{1}{q}\p_{a}[X_{[0,a+b]}(T)\geq x;J(T)]\\
&=\frac{1}{q}\Dpi^{-1}\left(\hat\p[\tau\leq T,X(\tau)\geq x;J(T)]+\hat\p[\tau>T,X(T)+a\geq x;J(T)]\right)^T\Dpi.
\end{align*}
Furthermore, using the fact that $\tau=\tau_x^+\wedge\tau_{x-a-b}^-$ we write 
\begin{align*}
&\p[\tau\leq T,X(\tau)\geq x;J(T)]=\p^q[\tau_x^+<\tau_{x-a-b}^{-};J(\tau_x^+)]\p[J(T)]\\&=W^q(a+b-x)W^q(a+b)^{-1}q(q\matI-Q)^{-1},
\end{align*}
where $W^q$ corresponds to a process with additional killing of rate $q$ in each state.
Moreover, 
\begin{align*}
&\p[\tau>T,X(T)\geq x-a;J(T)]=q\int_{x-a}^xu^q_{|x-a-b,x|}(y)\D y\\&=q\int_{x-a}^xW^q(a+b-x)W^{q}(a+b)^{-1}W^q(x-y)-W^q(-y)\D y,
\end{align*}
where in the second equality we use~\eqref{thm:14}.
Taking time-reversed quantities and using \eqref{eq:timerev_W} we arrive at
\begin{align*}
&U^q_{[-a,b]}[x-a,b]=(q\matI-Q)^{-1}W^q(a+b)^{-1}W^q(a+b-x)\\&+\int_{x-a}^xW^q(x-y)W^{q}(a+b)^{-1}W^q(a+b-x)-W^q(-y)\D y\\
&=-Z^q(a)F^q(0)^{-1}W^{q}(a+b)^{-1}W^q(a+b-x)-\int_{(x-a)^-}^0W^q(-y)\D y.
\end{align*}
Finally, taking $q\downarrow 0$ we obtain
\[U_{[-a,b]}[x,b]=-Z(a)F(0)^{-1}W(a+b)^{-1}W(b-x)-\int_{x^-}^0W(-y)\D y,\]
which has a density on $(-a,b)$ and point masses as given in the statement of the theorem.
\end{proof}

\begin{cor}\label{cor:reflecting}
 For a defective MAP it holds that
\begin{align}
u_{[-a,b|}(x)&=Z(a)Z(a+b)^{-1}W(b-x)-W(-x),\label{eq:cor1}\\
u_{|-a,b]}(x)&=W(a)W'_+(a+b)^{-1}W'(b-x)-W(-x),\label{eq:cor2}\\
u_{[-a,\infty)}(x)&=Z(a)F(0)^{-1}Re^{R(a+x)}-W(-x),\label{eq:cor3}\\
u_{(-\infty,b]}(x)&=-G^{-1}e^{G b}W'(b-x)-W(-x).\label{eq:cor4}
\end{align}
Additionally,
\begin{align*}
&U_{[-a,b|}\{-a\}=\matO, &U_{|-a,b]}\{b\}=W(a)W'_+(a+b)^{-1}W(0),\\
&U_{[-a,\infty)}\{-a\}=\matO, &U_{(-\infty,b]}\{b\}=-G^{-1}e^{G b}W(0).
\end{align*}
\end{cor}
In the case of a \levy process this result is due to~\cite{korolyuk_suprun,pistorius_exit}, but see also~\cite{doney}
for an alternative proof.
\begin{proof}[Proof of Corollary~\ref{cor:reflecting}]
Proof of~\eqref{eq:cor1}. \\
Note that for any Borel $A\subset[-a,b]$ we have
\[U_{[-a,b]}(A)=U_{[-a,b|}(A)+\p[J(\underline\sigma_{-a,b})]U_{[-a-b,0]}(A-b),\]
where $\underline\sigma_{-a,b}$ is the first passage time over level $b$ of a process reflected at $-a$. 
It is known that $\p[J(\underline\sigma_{-a,b})]=Z(a)Z(a+b)^{-1}$, see~\cite[Thm.\ 2]{ivanovs_scale}.
Hence for $x\in(-a,b)$
\begin{align*}
 &u_{[-a,b|}(x)=u_{[-a,b]}(x)-Z(a)Z(a+b)^{-1}u_{[-a-b,0]}(x-b)=\\
&=-Z(a)F(0)^{-1}W(a+b)^{-1}W'(b-x)-W(-x)\\
&Z(a)Z(a+b)^{-1}(Z(a+b)F(0)^{-1}W(a+b)^{-1}W'(b-x)+W(b-x))\\&=
Z(a)Z(a+b)^{-1}W(b-x)-W(-x).
\end{align*}
Moreover, $U_{[-a,b|}\{-a\}=\matO$ is implied by $U_{[-a,b]}\{-a\}=\matO$.
There is no point mass at $b$, since the process exceeds $b$ immediately after hitting it.

Proof of~\eqref{eq:cor2}. 
Similarly to the above, for $x\in(-a,b)$ we have
\[u_{[-a,b]}(x)=u_{|-a,b]}(x)+\p[J(\overline\sigma_{-a,b})]u_{[0,a+b]}(x+a),\]
where $\overline\sigma_{-a,b}$ is the first passage below $-a$ of a process reflected at $b$.
From \cite[Thm.\ 6]{ivanovs_scale} we find
\[\p[J(\overline\sigma_{-a,b})]=Z(a)+W(a)W'_+(a+b)^{-1}W(a+b)F(0)\]
and hence
\begin{align*}
 &u_{|-a,b]}(x)=-Z(a)F(0)^{-1}W(a+b)^{-1}W'(b-x)-W(-x)\\&(Z(a)+W(a)W'_+(a+b)^{-1}W(a+b)F(0))F(0)^{-1}W(a+b)^{-1}W'(b-x)\\
&=W(a)W'_+(a+b)^{-1}W'(b-x)-W(-x).
\end{align*}
Similarly we find that $U_{|-a,b]}\{b\}=W(a)W'_+(a+b)^{-1}W(0)$ and that there is no point mass at $-a$.

Finally, \eqref{eq:cor3} is obtained from~\eqref{eq:cor1} by taking the limit as $b\rightarrow\infty$ and using Lemma~\ref{lem:limits} and Corollary~\ref{cor:limits}.
In particular, we use $Z(a+b)^{-1}W(a+b)\rightarrow F(0)^{-1}R$ and $W(a+b)^{-1}W(b-x)\rightarrow e^{R(a+x)}$.
Moreover,~\eqref{eq:cor4} is obtained from~\eqref{eq:thm1} by letting $a\rightarrow\infty$, and the expressions for point masses follow similarly.
Let us remark that \eqref{eq:cor3} and \eqref{eq:cor4} can be obtained directly using a time-reversal argument similar to that in Lemma~\ref{lem:time_rev}.
\end{proof}

\section{Non-defective processes}\label{sec:non_def}
Let us discuss potential measures of a non-defective process.
One may kill the process at rate $q>0$ and then let $q\downarrow 0$ to retrieve the original process back.
It is known that all the quantities $F(0),G,R,H,W(x),W'(x),Z(x)$ are continuous in $q\geq 0$, and so the identities should still hold.
The problem is that some limiting matrices may be infinite or singular, and that potential measures can be infinite as well.
It is rather clear which cases are to be excluded, see also~\cite[Prop.\ XI.2.10]{APQ}.
\begin{cor}\label{cor:non_defective}
The results of Thm.~\ref{thm:absorbing} and Corollary~\ref{cor:reflecting} hold true for a non-defective process apart from the following cases.
 \begin{itemize}
\item $\mu<0$: exclude \eqref{eq:cor3};
\item $\mu=0$: exclude \eqref{thm:11},\eqref{eq:cor3},\eqref{eq:cor4};
  \item $\mu>0$: exclude \eqref{eq:cor4}.
 \end{itemize}
Moreover, when $\mu>0$ the term $F(0)^{-1}R$ in~\eqref{eq:cor3} should be understood in the limiting sense (as $q\downarrow 0$).
\end{cor}

Potential densities can also be used to obtain limiting distributions of reflected processes.
According to~\eqref{eq:intro1} the limiting distribution is given by \[\lim_{q\downarrow 0}qU_I^q(\D x),\]
and hence it can be proper only if the limiting measure is infinite. In certain sense these results will complement Corollary~\ref{cor:non_defective}.

Let us first establish a simple lemma. We recall that $\bpi$ is the stationary distribution of $J(t),t\geq 0$, and we let 
$\bpi_G$ denote the stationary distribution of $J(\tau_x^+),x\geq 0$ when $\mu\geq 0$.
\begin{lemma}\label{lem:limits2}
For a non-defective process every row of $-\lim_{q\downarrow 0}qF^q(0)^{-1}$ equals $\bpi$.
If in addition $\mu\geq 0$ then every row of $-\lim_{q\downarrow 0}q(G^q)^{-1}$ equals $\mu\bpi_G$.
\end{lemma}
\begin{proof}
Recall that $-qF^q(0)^{-1}=\p[J(e_q)]$, which immediately completes the proof of the first part. For the second part compute
\begin{equation}\label{eq:lim_exp}
-q(G^q)^{-1}=q\int_0^\infty e^{G^q x}\D x=\left(\e_iq\int_0^\infty \1{\tau_x^+<e_q,J(\tau_x^+)=j}\D x\right) 
\end{equation}
and observe that 
\[\int_0^\infty \1{\tau_x^+<e_q,J(\tau_x^+)=j}\D x/\overline X(e_q)\rightarrow (\bpi_G)_j \text{ a.s.}\]
Furthermore, it is known that $\overline X(t)/t\rightarrow \mu$ a.s.\ as $t\rightarrow\infty$, and also we can put $e_q=e_1/q$.
Hence the proof is complete if we can show that the limit as $q\downarrow 0$ can be taken under the expectation sign in~\eqref{eq:lim_exp}.
This follows by the extended dominated convergence theorem and the fact that $\lim_{q\downarrow 0}\e(q\overline X(e_q))=\mu$, which is easy to establish.
\end{proof}

We are ready to present results concerning limiting distributions.
\begin{cor}\label{cor:lim_distr}
The following limiting distributions exist, do not depend on the initial values $X(0)$ and $J(0)$, and are given by:
\begin{align*}
&\p_i[X_{[0,b]}(\infty)\in \D x,J(\infty)]=\bpi W(b)^{-1}(W'(b-x)\D x+W(0)\delta_b(\D x)),&\\
&\p_i[X_{[0,\infty)}(\infty)\in \D x,J(\infty)]=-\bpi Re^{R x}\D x, &\mu<0,\\
&\p_i[X_{(-\infty,0]}(\infty)\in \D x,J(\infty)]=\mu\bpi_G(W'(-x)\D x+W(0)\delta_0(\D x)), &\mu>0.\\
\end{align*}
\end{cor}
\begin{proof}
The fact that these limiting distributions exist and do not depend on the initial values can be established in various ways, for example we can use
the time-reversal argument, see Lemma~\ref{lem:time_rev}.
From Theorem~\ref{thm:reflecting} and~\eqref{eq:intro1} we see that 
\[\p[X_{[0,b]}(\infty)\geq x,J(\infty)]=-\lim_{q\downarrow 0}qF^q(0)^{-1}W(b)^{-1}W(b-x),\]
which according to Lemma~\ref{lem:limits2} yields
\[\p_i[X_{[0,b]}(\infty)\geq x,J(\infty)]=\bpi W(b)^{-1}W(b-x)\]
and then the first result follows.
The second and third result are obtained in a similar way using Corollary~\ref{cor:reflecting}. We remark that these results can be obtained directly from Lemma~\ref{lem:time_rev}. 
\end{proof}

Let us comment on the results of Corollary~\ref{cor:lim_distr}.
Firstly, the first result has the following alternative form
\[\p_i[X_{[-a,0]}(\infty)\in \D x,J(\infty)]=\bpi W(a)^{-1}(W'(-x)\D x+W(0)\delta_0(\D x)).\]
The second and third result are known in the literature but in a different form.
In particular, for $\mu<0$ it is known that
$(X_{[0,\infty)}(\infty)|J(\infty)=i)$ has a phase type distribution with transition rate matrix $\hat G$ started in state $i$, see e.g.~\cite[Cor.\ 2.23]{thesis}. In other words,
the column vector of densities is given by
\[\p(X_{[0,\infty)}(\infty)\in \D x|J(\infty))=e^{\hat G x}(-\hat G\one)\D x.\]
which leads to our result using Lemma~\ref{prop:time_rev}.
Considering the third result we compute the transform for large enough $\a>0$:
\[\int_{-\infty}^{0+} e^{\a x}\mu\bpi_G(W'(-x)\D x+W(0)\delta_0(\D x))=\mu\bpi_G\a F(\a)^{-1},\] 
where we use integration by parts showing that 
\begin{align}\label{eq:int_by_parts}
W(0)+\int_0^{\infty} e^{-\a x}W'(x)\D x=\a\int_0^{\infty} e^{-\a x}W(x)\D x=\a F(\a)^{-1}.
\end{align}
This confirms~\cite[Prop.\ 4.19]{thesis} building upon~\cite{asm_kella_multidim}.

\section{Concluding remarks}\label{sec:remarks}
\subsection{Exit from an open interval}
One may choose another way of terminating a process at a barrier. So far we have assumed that termination at $b$ occurs at the first exit from $(-\infty,b]$, 
and at $-a$ at the first exit from $[-a,\infty)$. One may choose to use open intervals $(-\infty,b)$ and $(-a,\infty)$ instead. 
Since $X(t)$ can not hit a level without immediately passing over it, the result of Theorem~\ref{thm:absorbing} is still true apart from the case when $a=0$,
i.e.\ the process is started at the boundary ($b=0$ does not cause problems). In this case the process is killed at time 0 leading to 0 measures.
Also $u_{[-a,b|}(x)$ is still given by~\eqref{eq:cor1}. There is however a substantial difference in the expression for $u_{|-a,b]}(x)$,
because there might be an excursion of $X(t)$ from its maximum of height $a+b$ with positive probability, see~\cite[Thm.\ 5]{ivanovs_scale}. 
The new expression is given by 
\[\lim_{c\uparrow a}u_{|-c,b]}(x)=W(a)W'_-(a+b)^{-1}W'(b-x)-W(-x),\]
where $W'_-$ denotes a left derivative. So the formula stays the same if the left and right derivatives of $W$ coincide at $a+b$.

\subsection{Known transforms}
There are some results available in the literature concerning transforms of a one-sided reflection of $X(t)$ at an exponential time.
In particular~\cite[Cor.\ 4.21]{thesis} states that 
\begin{align*}
&-\frac{1}{q}\e[e^{\a X(e_q)-(\a+\beta)\overline X(e_q)};J(e_q)]F^q(\a)=\matI+(\a+\beta)(G^q-\beta\matI)^{-1},\\ 
&-\frac{1}{q}F^q(\a)\e[e^{-\beta X(e_q)+(\a+\beta)\underline X(e_q)};J(e_q)]=\matI+(\a+\beta)(R^q-\beta\matI)^{-1}
\end{align*}
for $\a,\beta\geq 0$,
see also~\cite{dieker_mandjes} and~\cite{kypr_palmowski_fluct} for results of a similar type.
In particular, putting $\beta=0$ in the first equation and $\a=0$ in the second we obtain
\begin{align*}
&-\int_{-\infty}^{0+}e^{\a x}U_{(-\infty,0]}(\D x)F(\a)=\matI+\a G^{-1},\\ 
&-F(0)\int_{0-}^{\infty}e^{-\beta x}U_{[0,\infty)}(\D x)=\matI+\beta(R-\beta\matI)^{-1},
\end{align*}
where the common killing rate $q>0$ is implicit. 

Let us check~\eqref{eq:cor3} and \eqref{eq:cor4} against these results.
Firstly, for large enough $\a>0$ we compute similarly to~\eqref{eq:int_by_parts}
\begin{align*}
 \int_{-\infty}^0 e^{\a x}(-G^{-1}W'(-x)-W(-x))\D x-G^{-1}W(0)=-(\a G^{-1}+\matI)F(\a)^{-1}.
\end{align*}
Secondly, for $\beta\geq 0$ we compute
\begin{align*}
 &\int_0^{\infty} e^{-\beta x}F(0)^{-1}Re^{Rx}\D x=-F(0)^{-1}R(R-\beta\matI)^{-1}\\&=-F(0)^{-1}(\matI+\beta(R-\beta\matI)^{-1}),
\end{align*}
which confirms the above result.

\section*{Acknowledgments}
This work was supported by the Swiss National Science Foundation Project 200020-143889 and the EU-FP7 project Impact2C.

\bibliographystyle{plain}
\bibliography{Potential_measures.bib}
\end{document}